\documentclass[12pt,reqno, oneside]{amsart}
\usepackage{tikz}
\usepackage{graphicx}
\usepackage{pdfsync}  
\usepackage[height=9in, width = 6.4in]{geometry}
\usepackage{hyperref} %shold be loaded last

\newtheorem{thm}{Theorem}[section]
\newtheorem{cor}[thm]{Corollary}
\newtheorem{lem}[thm]{Lemma}

\newcommand{\St}{\tilde S}
\newcommand{\Gt}{\tilde G}
\newcommand{\Ht}{\tilde H}
\newcommand{\Sch}{Sch\"utzenberger}
\newcommand{\rn}{reverse-nonnegative}
\newcommand{\rp}{reverse-positive}

\usepackage{xcolor}
  {\color{#1}}{}

\begin{document}
\title{Lattice paths and the Geode}
\author{Ira M. Gessel$^*$}
\address{Department of Mathematics\\
   Brandeis University\\
   Waltham, MA 02453}
\email{gessel@brandeis.edu}
\date{July 17, 2025}
\begin{abstract}

Let $t_1,t_2,\dots$ be variables, and  let $S$ be the  formal power series in the variables $t_1, t_2,\dots$ satisfying 
\begin{math}
S=1+\sum_{n=1}^\infty t_n S^n. 
\end{math}
Let $S_1 =\sum_{n=1}^\infty t_n$. Wildberger and Rubine recently showed that  there is a formal power series $G$ in the $t_i$, which they called the ``Geode," satisfying 
$S=1+GS_1$.
In this paper we discuss some of the properties of the Geode and of the related series $H=G/S$, which satisfies $S=1/(1-HS_1)$. We show that 
\begin{equation*}
G=\biggl(1-\sum_{n=1}^\infty t_n (1+S+S^2+\cdots+S^{n-1})\biggr)^{-1},
\end{equation*}
and
\begin{equation*}
H=\biggl( 1-\sum_{n=2}^\infty t_n (S+S^2+\cdots+S^{n-1})\biggr)^{-1},
\end{equation*}
and we give combinatorial interpretations of $G$ and $H$ in terms of lattice paths.
\end{abstract}

\maketitle
%\tableofcontents

\thispagestyle{empty}
\section{Introduction}

Let $t_1, t_2,\dots$ be variables, and  let $S$ be the unique formal power series in these variables satisfying 
\begin{equation}
\label{e-S0}
S=1+\sum_{n=1}^\infty t_n S^n. 
\end{equation}
Let $S_1 =\sum_{n=1}^\infty t_n$. Wildberger and Rubine \cite[Theorem 12]{wr} showed that $S-1$ is divisible by $S_1$; i.e., there is a formal power series $G$ in the $t_i$, which they called the ``Geode," satisfying 
$S=1+GS_1$. (See also Rubine \cite{rubine}.)

In this paper we discuss some of the properties of the Geode. In Section \ref{s-formulas}, we show that 
\begin{equation*}
G=\biggl(1-\sum_{n=1}^\infty t_n (1+S+S^2+\cdots+S^{n-1})\biggr)^{-1},\label{e-G20}
\end{equation*}
which gives another proof that $G$ is a power series in $t_1,t_2,\dots$ with nonnegative integer coefficients.%
\footnote{Wildberger and Rubine's inductive proof may be sketched as follows: Let $S_n$ be the sum of all terms in $S$ of degree $n$, where the degree of a monomial $t_1^{i_1}\cdots t_{k}^{i_k}$ is $i_1+\cdots + i_k$. Then $S_1=t_1+t_2+\cdots$, in agreement with our previous definition of $S_1$, and for $n\ge2$, equating terms of degree $n$ on both sides of \eqref{e-S0} shows that $S_n$ is a sum of products, each of which is divisible by some $S_i$ with $1\le i <n$.
}
 We also introduce a related power series $H$, which satisfies $G=SH$ and 
$S=1/(1-H S_1)$, and we show that
\begin{equation*}
\label{e-H20}
H=\biggl( 1-\sum_{n=2}^\infty t_n (S+S^2+\cdots+S^{n-1})\biggr)^{-1},
\end{equation*}
which implies that $H$ has nonnegative coefficients.

In Section \ref{s-paths} we give combinatorial interpretations to the series $S$, $G$, and $H$, and to the formulas relating them, in terms of lattice paths. (Wildberger and Rubine worked with polygon dissections rather than lattice paths.) We represent  the paths by sequences of integers greater than or equal to $-1$, where the weight of $n\ge0$ is $t_{n+1}$ and the weight of $-1$ is $1$. Then $S$ counts sequences with sum 0 and every partial sum nonnegative (excursions), $G$ counts sequences with every partial sum nonnegative, and $H$ counts sequences with every nonempty partial sum positive. 

In Section \ref{s-cat} we look at what happens when we set $t_n=0$ for $n>2$. Here \eqref{e-S0} becomes a quadratic equation which can be solved explicitly, and we obtain generating functins for Catalan, Motzkin, Riordan, and Schr\"oder numbers.

In Section \ref{s-app} we gave an indirect alternative proof of the lattice path interpretation of the Geode, using one of the standard techniques of lattice path enumeration.
 
\section{Formulas for the Geode}
\label{s-formulas}

We first show the existence of a power series $G$ in the variables $t_1,t_2, t_3, \dots$ satisfying $S = 1+GS_1$. 
(Wildberger and Rubine did not include the variable $t_1$ as it is not so relevant to polygon dissections.)

Since the ring of formal power series in $t_1,t_2,\dots$ (over the integers or rationals, for example) is an integral domain, we may work in its field of fractions. Thus we may always divide by a nonzero power series, even though the quotient is not guaranteed to be a power series.
We  define $G$ to be $(S-1)/S_1$, so $S=1+S_1G$ and we define $H$ to be $G/S$. 

\begin{thm}
\label{t-1}
We have
\begin{align}
S &= \biggl(1-\sum_{n=1}^{\infty} t_nS^{n-1}\biggr)^{-1}\label{e-S1}\\
  &= \frac{1}{1-H S_1}, \label{e-SH}
\end{align}
\begin{equation}
G=\biggl(1-\sum_{n=1}^\infty t_n (1+S+S^2+\cdots+S^{n-1})\biggr)^{-1},\label{e-G2}
\end{equation}
and 
\begin{equation}
\label{e-H2}
H=\biggl( 1-\sum_{n=2}^\infty t_n (S+S^2+\cdots+S^{n-1})\biggr)^{-1}.
\end{equation}
\end{thm}

\begin{proof}
Equation \eqref{e-S1} follows easily from \eqref{e-S0}. For \eqref{e-SH}, we have
\begin{equation*}
1-HS_1 = 1-G S_1/S = 1-(S-1)/S =1/S.
\end{equation*}

From \eqref{e-S0}
we have
$S =1+ S_1 + \sum_{n=1}^\infty t_n (S^n-1)$, so
\begin{align}
G&=(S-1)/S_1=
 1+\sum_{n=1}^\infty t_n \left(\frac{S^n-1}{S_1}\right)\notag\\
&=1+\sum_{n=1}^\infty t_n \frac{S-1}{S_1}(1+S+S^2+\cdots+S^{n-1})\notag\\
&= 1+G\sum_{n=1}^\infty t_n (1+S+S^2+\cdots+S^{n-1}).\label{e-G1}
\end{align}
Then \eqref{e-G2} follows directly from \eqref{e-G1}..
Given \eqref{e-G2}, to prove \eqref{e-H2} it is sufficient to show that 
$H^{-1}-G^{-1}=S_1$. From the definition of $H$,
 we have  $H^{-1}=S/G$, so $H^{-1}-G^{-1}=(S-1)/G = S_1$.
 \end{proof}
 
Equation \eqref{e-G2} leads to a short proof of one of the conjectures of Wildberger and Rubine \cite[Section 11]{wr}.
Their conjecture is that 
\begin{equation*}
G(0,-f,f,\dots, -f,f) = \sum_{n=0}^\infty k^n f^n,
\end{equation*}
where there are $2k+1$ parameters in $G(0,-f,f,\dots, -f,f)$ and $G(u_1,\dots u_m)$ is the result of setting $t_i= u_i$ for $1\le i\le m$ and $t_i=0$ for $i>m$ in $G$. (Our notation differs slightly from Wildberger and Rubine's because we have including the variable $t_1$ and they have not.) Note that with this substitution $S_1=0$ and thus the formula $S=1+GS_1$ gives $S=1$, but is no help in evaluating $G$.

We have the following corollary of Theorem \ref{t-1}.

\begin{cor}
Suppose that $u_1+u_2+\cdots + u_m = 0.$ Then
\begin{equation}
\label{e-Gu}
G(u_1,\dots, u_m) = \biggl(1-\sum_{n=1}^m nu_n\biggr)^{-1},
\end{equation}
assuming that the right side of \eqref{e-Gu} exists as a formal power series, and $H(u_1,\dots, u_m) = G(u_1, \dots, u_m)$.
\end{cor}

\begin{proof}
If $u_1+\cdots+u_m=0$ then $S_1(u_1,\dots, u_m) = 0$, so since $S = 1+GS_1$, we have $S(u_1,\dots, u_m) = 1$. Then \eqref{e-Gu}  follows from \eqref{e-G2}. 
Since $H=G/S$ and $S(u_1,\dots, u_m)=1$, we have $H(u_1,\dots, u_m) = G(u_1,\dots, u_m)$.
\end{proof}

For Wildberger and Rubine's conjecture, we take $u_1=0$ and $u_n=(-1)^{n-1} f$ for $2\le n\le 2k+1$.
Then 
\begin{equation*}
\sum_{n=1}^{2k+1} nu_n = \sum_{n=2}^{2k+1} n(-1)^{n-1}f = kf,
\end{equation*}
so $G(0,-f,f,\dots, -f,f) = (1-kf)^{-1}$. 

A different proof of this conjecture was given by Amdeberhan and Zeilberger \cite[Theorem 3.1]{az}.

\section{Lattice paths}
\label{s-paths}

In this section, we interpret the formulas of Theorem \ref{t-1} in terms of lattice paths. Our basic tool is \emph{factorization} of paths. For some other applications of factorization of lattice paths, see Labelle and Yeh \cite{ly}, Lemus Vidales \cite{cristobal}
Gessel \cite{imgfact, imglag}, and Merlini et al. \cite{mrsv}.
\subsection{Paths and factorization}

We  define a path formally to be a finite sequence of integers. We represent $(s_1,s_2,\dots, s_k)$ geometrically as the path in the plane with vertices $(0,s_1), (1,s_1+s_2), \dots, (k, s_1+s_2+\cdots +s_k)$. Thus, for example, the path $(1,1,-3,2)$ is represented geometrically as in Figure \ref{f-geom1}.
\begin{figure}[h]
\begin{tikzpicture}
\centering
\draw[help lines] (0,0) grid (4,3);
\draw[thick] (0,1) -- (1,2) -- (2,3) -- (3,0) -- (4,2);
\end{tikzpicture}
\caption{Geometric depiction of the path $(1,1,-3,2)$.}
\label{f-geom1}
\end{figure}
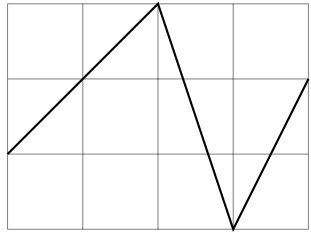

The geometric representations of paths by default start at the origin (so our figures do not show coordinate axes). The entries $s_i$ in a path $(s_1,s_2,\dots, s_k)$ are called the \emph{steps} of the path. 
We will often refer to paths geometrically rather than using the formal definition; for example, if we say that a path $(s_1, \dots, s_k)$ ends on the $x$-axis, we mean that $s_1+\cdots + s_k=0$.

If $P=(p_1,\dots, p_m)$ and $Q=(q_1,\dots, q_n)$ are paths, then their \emph{product} $PQ$ is the path $(p_1,\dots, p_m, q_1, \dots, q_n)$. Geometrically, $PQ$ is obtained by translating $Q$ to start at the endpoint of $P$.

The \emph{reversal} of a path $(s_1,\dots, s_k)$ is the path $(-s_k, \dots, -s_1)$. Geometrically, reversing a path is reflecting it in a vertical axis.
We call a path \emph{nonnegative} if it never goes below the $x$-axis; i.e.,  every partial sum is nonnegative. (Nonnegative paths are sometimes called \emph{meanders}.) We call a path \emph{positive} if it stays above the $x$ axis; i.e., every nonempty partial sum is positive. Note that the empty path is a positive path.  We call a path \emph{\rn} if it is the reversal of a nonnegative path and \emph{\rp} if it is the reversal of a positive path. (See Figure \ref{f-rev}.) Then \rn\ paths are paths whose endpoints are lowest points, and  \rp\ paths are paths in which the endpoint is strictly lower than every other point.

\begin{figure}[h]
\begin{tikzpicture}
\draw[help lines] (0,0) grid (4,2);
\draw[thick] (0,0) -- (1,1) -- (2,2) -- (3,0) -- (4,1);
\end{tikzpicture}
\hskip 30pt
\begin{tikzpicture}
\draw[help lines] (0,0) grid (4,2);
\draw[thick] (0,1) -- (1,0) -- (2,2) -- (3,1) -- (4,0);
\end{tikzpicture}
\caption{The nonnegative path $(1,1,-2,1)$ and its reversal $(-1,2,-1,-1)$.}
\label{f-rev}
\end{figure}
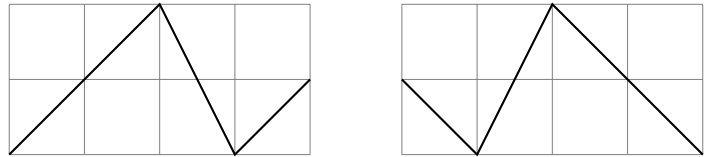

The \emph{height} of point $(m,n)$ is $n$, so the height of the endpoint of the path $(s_1,\dots, s_k)$ is $s_1+\cdots + s_k$.
An \emph{excursion} is a path (possibly empty) that ends at height~0 and never goes below height 0. (See Figure \ref{f-exc}.)) Thus an excursion is a path that is both nonnegative and \rn.
\begin{figure}[h]
\begin{tikzpicture}
\centering
\draw[help lines] (0,0) grid (5,2);
\draw[thick] (0,0) -- (1,1) -- (2,0) -- (3,2) -- (4,1) -- (5,0);
\end{tikzpicture}
\caption{An excursion.}
\label{f-exc}
\end{figure}
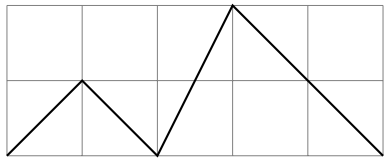

From here on, we assume that all steps are in the set $\{-1,0,1,2,\dots\}$. We call nonnegative steps \emph{up steps} (so $0$ is considered an up step) and we call $-1$ the \emph{down step}. It will be convenient in representing factorizations of paths by formulas to denote the up step $n$ by $U_n$ and the down step $-1$ by $D$. The restriction on the step set implies that \rn\ and \rp\ paths have  very simple decompositions, but nonnegative and positive paths, which are closely related to the Geode, are more complicated.

Next, we describe the decompositions for \rn\ and \rp\ paths that are helpful in deriving the combinatorial interpretations for $S$, $G$ and $H$.
\begin{lem} 
\label{l-1}
Let $P$ be a \rn\ path that ends at height $-n$. 
Then $P$ can be factored uniquely as $E_1 D E_2 D \cdots DE_{n+1}$ where each $E_i$ is an excursion and $D$ is the down step $-1$.
\end{lem}

\begin{proof} 
We proceed by induction on $n$. If $n=0$ then $P$ is an excursion. If $n>0$ then $P$ must have a first down step from height 0 to height $-1$. (This is where we use the assumption that there are no steps less than $-1$.) Cutting $P$ before and after this down step shows that $P=E_1 D Q$ where $Q$ is a \rn\ path that, when translated to start at the origin, ends at height $-(n-1)$. We then apply induction to $Q$.
We leave it to the reader to verify that this factorization is unique.
\end{proof}

See Figure \ref{f-l-1} for an example of the factorization of Lemma \ref{l-1} with $n=2$.

In Lemma \ref{l-1}, and similarly in other factorization statements, it is  understood that the converse is also true, in the sense that if $P=E_1 D E_2 D \cdots DE_{n+1}$ where each $E_i$ is an excursion, then $P$ is a \rn\ path that ends at height $-n$.

\begin{figure}[h]
\begin{tikzpicture}[scale=.6]
\centering
\draw[help lines] (0,0) grid (13,4);
\draw[thick] (0,2) -- (1,4) -- (2,3) -- (3,2);
\draw[thick, dashed] (3,2) -- (4,1);
\draw[thick]
  (4,1) -- (5,2) -- (6,2) -- (7,1);
  
\draw[thick, dashed] (7,1) -- (8,0);
\draw[thick] (8,0) -- (9,2) -- (10,1) -- (11,0)
  -- (12,1) -- (13,0);
\end{tikzpicture}
\caption{The factorization of Lemma \ref{l-1}.}
\label{f-l-1}
\end{figure}
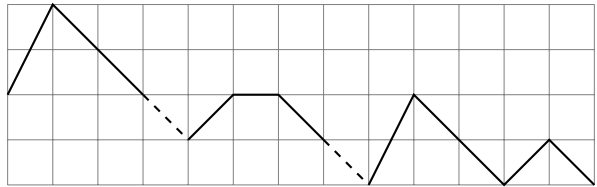

We have a variant of Lemma \ref{l-1} that will also be needed later.
%The proof is similar to that of Lemma \ref{l-1} and is omitted.

\begin{lem}
\label{l-2}
Let $P$ be a \rp\ path that ends at height $-n$.
Then $P$ can be factored uniquely as $E_1 D E_2 D \cdots E_n D$ where each $E_i$ is an excursion and $D$ is the down step $-1$. \textup(If $n=0$ this means that $P$ is the empty path.\textup)
\end{lem}
\begin{proof}
The case $n=0$ is clear, so suppose that $n>0$.
Since $P$ satisfies the hypothesis of Lemma \ref{l-1}, it can be factored uniquely as $E_1 D E_2 D \cdots DE_{n+1}$ as in
Lemma \ref{l-1}. But since  $P$ stays above  height $-n$ until the end, it must end with a down step, so $E_{n+1}$ is empty.
\end{proof}

\subsection{Generating functions}
Let us assign a weight to every step, and assign to each path the product of the weights of its steps. Then the \emph{generating function} (also called \emph{generating series} or \emph{counting series}) for a set of paths is the sum of its weights, where we assume that the sum exists as a formal power series.  It is clear that if $A$ and $B$ are disjoint sets of paths then the generating function for $A\cup B$ is the sum of the generating functions of $A$ and $B$, and if every element of $AB$ has a unique factorization $\alpha\beta$ with $\alpha\in A$ and $\beta\in B$ then the generating function for $AB$ is the product of the generating functions of $A$ and $B$.

For our applications, we take the weight of each up step $n\ge0$ to be $t_{n+1}$ and we take the weight of the  down step $-1$ to be 1. 
(We weight the up step $n$ by $t_{n+1}$  rather than $t_n$ to agree with the notation of Wildberger and Rubine \cite{wr}; this also makes the defining equation \eqref{e-S0} look nicer.)

\begin{thm}
\label{t-S}
The series $S$ is the generating function for excursions.
\end{thm}

\begin{proof}
Let $\St$ be the generating function for excursions. A path counting by $\St$ is either empty, with weight 1, or starts with an up step. Suppose that such a path $P$ starts with the up step $n-1$.  Then by Lemma \ref{l-1}, $P$ may be factored uniquely as  $U_{n-1}E_1DE_2D\cdots E_n$ 
where each $E_i$ is an excursion.
Thus the sum of the weights of all excursions starting with the up step $n-1$ is $t_n \St^n$, so summing on $n$ gives
$\St= 1+\sum_{n=1}^\infty t_n {\St}^n$. Then by \eqref{e-S0} and the uniqueness of the solution of this equation,  we have $\St=S$.
\end{proof}

\begin{thm} 
\label{t-Gthm}
The Geode  $G$ is the generating function for  nonnegative paths.
\end{thm}
\begin{proof}
Let $\Gt$ be the generating function for nonnegative paths. We will show that $S-1=\Gt S_1$, which implies that $\Gt=G$.

Let $E$ be a nonempty excursion. Then $E$ contains a last up step (recall that~$0$ is an up step).
Removing this last up step and the following down steps leaves a nonnegative path. Conversely, we obtain all nonempty excursions 
uniquely by starting with a nonnegative path, appending an arbitrary up step, and then appending enough down steps to bring the path down to the $x$-axis. Since down steps have weight 1, we see that $S-1=\Gt S_1$.
\end{proof}

See Figure \ref{f-Gfac} for an example of the construction described in this proof.

\begin{figure}[h]
\begin{tikzpicture}
\centering
\draw[help lines] (0,0) grid (8,3);
\draw[thick] (0,0) -- (1,2) -- (2,1) -- (3,0) -- (4,1);
\draw[thick,dotted] (4,1) -- (5,3);
\draw[thick,dashed] (5,3) -- (6,2) -- (7,1) -- (8,0);
\fill (4,1) circle (2pt);
\fill (5,3) circle (2pt);
\end{tikzpicture}
\caption{An excursion obtained from a nonnegative path.}
\label{f-Gfac}
\end{figure}
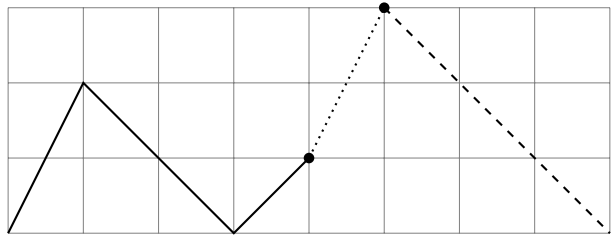

Note that the proof of Theorem \ref{t-Gthm} depends critically on down steps having weight~1; there does not seem to be such a simple formula for counting nonnegative paths where we keep track of the endpoint.

\begin{thm} 
The  series $H$ is the generating function for positive paths.
\end{thm}
\begin{proof}
Let $\Ht$  be the generating function for positive paths.
Every nonnegative path may be factored uniquely as a possibly empty excursion followed by a possibly empty positive path. (See Figure \ref{f-label2} for an example.)
Thus $G=S\Ht$ so $\Ht = G/S=H$.
\end{proof}

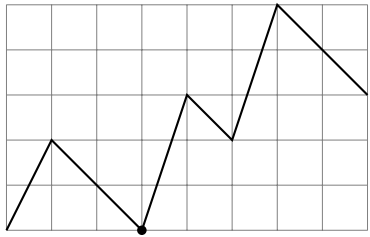
\begin{figure}[h]
\begin{tikzpicture}[scale=.6]
\centering
\draw[help lines] (0,0) grid (8,5);
\draw[thick] (0,0) -- (1,2) -- (2,1) -- (3,0) -- (4,3) -- (5,2) -- (6,5) -- (7,4) -- (8,3);
\fill (3,0) circle (3pt);
\end{tikzpicture}
\caption{Factorization of a nonnegative path as an excursion followed by a positive path.}
\label{f-label2}
\end{figure}

\subsection{Free monoids}
To give combinatorial interpretations to the formulas of Theorem \ref{t-1}, it is helpful to use the concept of free monoids.
A \emph{monoid} is a set with an associative binary operation and an identity element. 
We call an element $x$ of a monoid \emph{irreducible} if it is not the identity and cannot be factored as $yz$ where neither $y$ nor $z$ is the identity. If every element of a monoid $M$ has a unique factorization into irreducibles then $M$ is called a \emph{free monoid}. If $M$ is a free monoid, we will call its irreducible elements \emph{primes}. It is clear that the set of all paths (with any fixed set of steps), with the operation of concatenation, is a free monoid, in which the primes are the steps. As a more interesting example, the set of excursions is a free monoid, in which the primes are nonempty excursions that return to the $x$-axis for the first time at the end. These are  called \emph{arches}.

Free monoids are useful in working with generating functions for the following reason. Let $M$ be a free monoid. Suppose that a weight function is defined on $M$ so that the weight of an element of $M$ is the product of the weights of its constituent primes. Then if $U$ is the sum of the weights of all elements of $M$ and $V$ is the sum of the weights of all the primes of $M$, then
\begin{equation}
\label{e-fm1}
U = \sum_{n=0}^\infty V^n = \frac{1}{1-V},
\end{equation}
assuming that the sum converges as a formal power series. 
All of the formulas of Theorem \ref{t-1} will be seen to be instances of \eqref{e-fm1}.

To prove one of these formulas we must first identify the set of paths counted by the left side, prove that it is a free monoid (in these cases it will obviously be a monoid),  determine the primes, and find their generating function.

Although the proofs of freeness of the monoids we consider are simple enough that we could use ad hoc arguments, there is a helpful lemma that makes the proofs easier.

Let us say that a submonoid $M$ of a free monoid $F$ satisfies \emph{Sch\"utzenberger's criterion} if it has the property that for every $p$, $q$, and $r$ in $F$, if $p$, $pq$, $qr$, and $r$ are in $M$ then $q$ is in $M$. (In our applications $F$ is the free monoid of all paths.)

Then we have the following lemma due to Sch\"utzenberger \cite[Theorem 1.4]{schutzenberger}; see also Gessel and Li \cite[Lemma 3]{gl} for a conveniently accessible proof.  (We will  use only the sufficiency of Sch\"utzenberger's criterion.)
\begin{lem}
\label{l-schutz}
Let $M$ be a submonoid of a free monoid.
Then $M$ is free if and only if $M$ satisfies Sch\"utzenberger's criterion.
\end{lem}

It is easy to see directly that excursions form a free monoid, in which the primes are arches, but we can also use Sch\"utzenberger's criterion: if $p$ and $pq$ are excursions then $q$ is an excursion, so \Sch's criterion is satisfied for excursions.

We want to show that the set of nonnegative paths and the set of positive paths also form free monoids. If $M$ is either of these sets then $M$ is clearly a monoid and  if $qr\in M$ then $q\in M$, so  $M$ satisfies \Sch's criterion.

Similarly, the set of reverse-nonnegative paths and the set of reverse-positive paths are free monoids. Lemma \ref{l-2} shows that the prime reverse-positive paths are the paths of the form $ED$, where $D$ is an excursion. On the other hand, Lemma \ref{l-1} does not describe the prime reverse-nonnegative paths; we leave it to the curious reader to determine what they are. 

We are now ready to explain the combinatorial interpretations of the formulas of Theorem \ref{t-1}.
For \eqref{e-S1} we need to show that the generating function for arches is $\sum_{n=0}^\infty t_{n+1} S^{n}$. The proof is   similar to the proof of Theorem \ref{t-S}. An arch must start with an up step $n$ for some $n\ge0$, followed by a reverse-positive path ending at height $-n$. Then by Lemma \ref{l-2}, it can be factored as $U_n E_1 D E_2 D \cdots E_n D$. This gives the generating function for arches as
$\sum_{n=0}^{\infty} t_{n+1}S^{n}$.

For \eqref{e-SH} we need to show that the generating function for arches may also be expressed as $H S_1$. The proof is very similar to that of Theorem \ref{t-Gthm}: every arch can be obtained uniquely by starting with a positive path,  appending an arbitrary up step, and then appending enough down steps to bring the path down to the $x$-axis.

%:here: fix the following — reverse-positive and reverse-nonnegative.

For \eqref{e-G2} we need to describe a decomposition for prime nonnegative paths. Let $P$ be a prime nonnegative path and suppose that $P$ starts with the up step $n$ so that $P$ may be factored as $U_nQ$ for some path $Q$. 
%If $Q$ is empty, then $P=U_n$. ,
Then $P$ ends at height less than or equal to $n$ since if not, the factors $U_n$ and $Q$ would both be nonnegative and nonempty so $P$ would not be a prime. Suppose that $P$ ends at height $j$, where $0\le j \le n$. Then there cannot be any points of $P$ 
at a height less than or equal to $j$ other than the starting and ending points, since otherwise cutting $P$ at such a point would factor it into two nonempty nonnegative paths. Thus $Q$ is a reverse-positive path, and we may apply Lemma \ref{l-2} to $Q$ to see that $Q$ can be factored uniquely as 
$E_1D E_2 D \cdots E_{n-j}D$ where each $E_i$ is an excursion. (If $j=n$ then $Q$ is the empty path.) Thus the generating function for prime nonnegative paths with these values of $n$ and $j$ is $t_{n+1} S^{n-j}$. Summing over $j$ from 0 to $n$  gives $t_{n+1}(1+S+S^2 +  \cdots + S^{n})$, so the generating function for prime nonnegative paths is 
$\sum_{n=0}^\infty t_{n+1}(1+S+S^2 + S^n)$.

%
%
%For \eqref{e-H2} we describe a decomposition for prime positive paths. We omit  details, as this decomposition is similar to that for prime nonnegative paths. Every prime positive path starting with the up step $n$ can be  factored uniquely as 
%$U_{n} E_1 D E_2 D \cdots D E_{n-j}$, where each $E_i$ is an excursion and $1\le j\le n$. Thus the generating function for prime positive paths is 
%$\sum_{n=1}^\infty t_{n+1}(S+S^2 +\cdots+ S^n)$.

For \eqref{e-H2} we describe a decomposition for prime positive paths. We omit  details, as this decomposition is similar to that for prime nonnegative paths. Every prime positive path starting with the up step $n$ ends at height $j$ for some $j$ 
with $1\le j\le n$ and can be factored as $U_nQ$, where $Q$ is reverse-nonnegative path. By Lemma \ref{l-1}, $Q$ can be factored as 
$E_1 D E_2 D \cdots D E_{n-j+1}$, where each $E_i$ is an excursion.  Thus the generating function for prime positive paths is 
$\sum_{n=1}^\infty t_{n+1}(S+S^2 +\cdots+ S^n)$.

See Figure \ref{f-primepos} for an example of the factorization of a prime positive path with $n=3$ and  $j=1$. Note that this path is not a prime nonnegative path.
\begin{figure}[h]
\begin{tikzpicture}
\centering
\draw[help lines] (0,0) grid (6,4);
\draw[thick, dotted] (0,0) -- (1,3);
\draw[thick] (1,3) -- (2,4) -- (3,3);
\draw[thick, dashed] (3,3) -- (4,2);
\draw[thick] (4,2) -- (5,3) -- (6,2);
\end{tikzpicture}
\caption{A prime positive path.}
\label{f-primepos}
\end{figure}
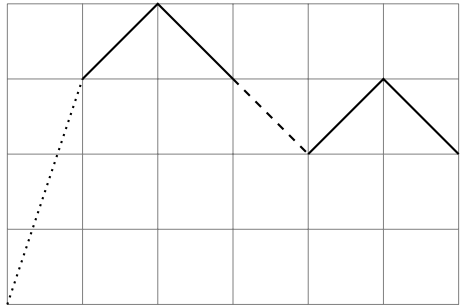

\section{Catalan, Motzkin, Riordan, and Schr\"oder numbers}
\label{s-cat}
The identities in Theorem \ref{t-1} and their combinatorial interpretations are of some interest even in the case in which 
$t_n=0$ for $n>2$. In this case we have the explicit formulas
\begin{gather}
S = \frac{1-t_1 -\sqrt{(1-t_1)^2 -4t_2}}{2t_2}\label{e-S12}\\[3pt]
H =\frac{1+t_1-\sqrt{(1-t_1)^2 -4t_2}}{2(t_1+t_2)}\label{e-H12}\\[3pt]
G=SH = \frac{1-t_1 -2t_2 -\sqrt{(1-t_1)^2 -4t_2}}{2t_2(t_1+t_2)}=(H-1)/t_2.
\end{gather}
The equation $G=(H-1)/t_2$ is easily explained combinatorially: A nonempty positive path with steps $-1$, $0$, and $1$ must consist of an up step $1$ followed by a nonnegative path. This equation also implies that the coefficients of $G$ are essentially the same as the coefficients of $H$, so we discuss only  the coefficients of $H$ below.

If we set $t_1=0$ and  $t_2=x$ then
\begin{equation*}
S=H = \frac{1-\sqrt{1-4x}}{2x},
\end{equation*}
the generating function for the Catalan numbers.

Our combinatorial interpretation for $S$ reduces to the well-known combinatorial interpretation of the Catalan numbers as counting Dyck paths. The combinatorial interpretation to the Catalan numbers given by Theorem \ref{t-Gthm} in this case is number~28 in Stanley's collection of Catalan number interpretations \cite{stanley}. Stanley attributes the result to Emeric Deutsch and gives essentially the same proof as our proof of Theorem \ref{t-Gthm} in this case.

If we set $t_1=x$ and $t_2=x^2$ in Equations \eqref{e-S12} and \eqref{e-H12} we get
\begin{equation*}
 S= \frac{1-x -\sqrt{1-2x-3x^2}}{2x^2},
\end{equation*}
the generating function for the Motzkin numbers, and
\begin{equation*}
H = \frac{1+x-\sqrt{1-2x-3x^2}}{2x(1+x)},
\end{equation*}
the generating function for Riordan numbers. 

If we set $t_1=x$ and $t_2=x$ then we have
\begin{equation*}
S = \frac{1-x-\sqrt{1-6x+x^2}}{2x},
\end{equation*}
the generating function for the large Schr\"oder numbers,
\[H= \frac{1+x-\sqrt{1-6x+x^2}}{4x},\]
the generating function for the small Schr\"oder numbers.

\section{Appendix: Another proof of Theorem \ref{t-Gthm}}
\label{s-app}

In this section we present another proof of Theorem \ref{t-Gthm}, which states that the Geode $G$ is the generating function for nonnegative paths. This proof is less straightforward than the one in Section \ref{s-paths} but is an interesting application of lattice path enumeration techniques. It proceeds by first using an indirect method to count nonnegative paths with a nontrivial weight on down steps, and then carefully setting the weight for down steps  to~$1$.

We start by recalling a well-known factorization for lattice paths, sometimes called the Wiener-Hopf factorization, which applies to paths with arbitrary steps.  We can factor a path into three parts (each possibly empty) by cutting the path at the first lowest point and at the last lowest point. 
The first part is a reverse-positive path, the second part is an excursion, and the third part is a positive path.
An example is shown in Figure \ref{f-wh}. This factorization is quite useful in lattice path enumeration. See \cite{analytic,bf,pathology,imgfact,cristobal} for some applications.
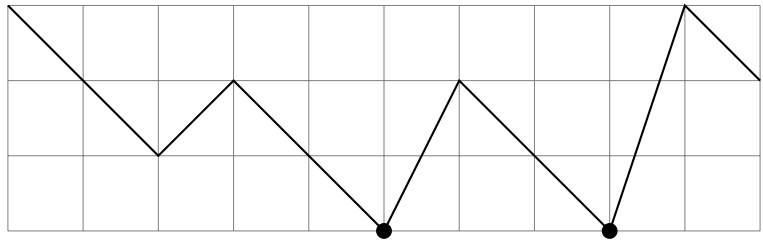
\begin{figure}[h]
\begin{tikzpicture}%[scale=.6]
\centering
\draw[help lines] (0,0) grid (10,3);
\draw[thick] (0,3) -- (1,2) -- (2,1) -- (3,2) -- (4,1) -- (5,0) -- (6,2) -- (7,1) -- (8,0) -- (9,3) -- (10,2);
\fill (5,0) circle (3pt);
\fill (8,0) circle (3pt);
\end{tikzpicture}
\caption{The Wiener-Hopf factorization of a path}
\label{f-wh}
\end{figure}

For our problem, we want to combine the second and third parts into a nonnegative path, so we have a factorization of every path into a \rp\ path followed by a nonnegative path. 

The set of all paths is easy to count, and Lemma \ref{l-2} allows us to count \rp\ paths. However, there is a complication. 
There is no problem in assigning down steps the weight 1 when counting nonnegative paths, since for any finite multiset of up steps, there are only finite many nonnegative paths using these up steps, no matter how many down steps they have. But for arbitrary paths, and for \rp\ paths, there are infinitely many paths with a given multiset of up steps. So to get a generating function identity from the factorization of arbitrary paths into \rp\ paths and nonnegative paths, we change our definition of the generating function for a set of paths: we keep the weight of $t_{i+1}$ for the up step $i\ge0$, but we change the weight of the down step $-1$  from 1 to a new variable $y$. Let $S(y)$ be the new generating function for excursions, so $S(1) = S$ by Theorem \ref{t-S}. Let $N(y)$ be the new generating function for nonnegative paths. Then to prove Theorem \ref{t-Gthm} we will show that $N(1) = G$.

By Lemma \ref{l-2} the new generating function for \rp\ paths is $\bigl(1-S(y)y\bigr)^{-1}$. The new generating function for all paths is $(1-y-t_1-t_2-\cdots)^{-1}=(1-y-S_1)^{-1}$. Then the Wiener-Hopf factorization  gives the identity
\begin{equation}
\label{e-N}
\frac{1}{1-y-S_1}=\frac{1}{1-S(y)y}N(y).
\end{equation}
We would like to determine $N(1)$ from \eqref{e-N}. We cannot set $y=1$ in \eqref{e-N}, since the result would not converge as a formal power series in $t_1, t_2,\dots$. 
But this problem is easy to fix. We replace \eqref{e-N} with the equivalent equation
\begin{equation*}
1-S(y)y = (1-y-S_1)N(y).
\end{equation*}
Now there is no problem in setting $y=1$ and we find that $1-S =-S_1 N(1)$, so $N(1)= (S-1)/S_1=G$, as desired.

\bigskip
\textbf{Acknowledgment.} I would like to thank Dean Rubine and Martin Rubey for pointing out typos in earlier versions of this paper.

%for BibTex
\bibliography{geode}{}

\providecommand{\bysame}{\leavevmode\hbox to3em{\hrulefill}\thinspace}
\begin{thebibliography}{10}

\bibitem{az}
Tewodros Amdeberhan and Doron Zeilberger, \emph{Proofs of three {G}eode
  conjectures}, \arXiv{2506.17862}{math.CO}.

\bibitem{analytic}
Andrei Asinowski, Axel Bacher, Cyril Banderier, and Bernhard Gittenberger,
  \emph{Analytic combinatorics of lattice paths with forbidden patterns, the
  vectorial kernel method, and generating functions for pushdown automata},
  Algorithmica \textbf{82} (2020), no.~3, 386--428.

\bibitem{bf}
Cyril Banderier and Philippe Flajolet, \emph{Basic analytic combinatorics of
  directed lattice paths}, Theoret. Comput. Sci. \textbf{281} (2002), no.~1-2,
  37--80.

\bibitem{pathology}
Cyril Banderier, Marie-Louise Lackner, and Michael Wallner,
  \emph{Latticepathology and symmetric functions \textup(extended
  abstract\textup)}, 31st {I}nternational {C}onference on {P}robabilistic,
  {C}ombinatorial and {A}symptotic {M}ethods for the {A}nalysis of
  {A}lgorithms, LIPIcs. Leibniz Int. Proc. Inform., vol. 159, Schloss Dagstuhl.
  Leibniz-Zent. Inform., Wadern, 2020, Art. No. 2, 16 pp.

\bibitem{imglag}
Ira Gessel, \emph{A noncommutative generalization and {$q$}-analog of the
  {L}agrange inversion formula}, Trans. Amer. Math. Soc. \textbf{257} (1980),
  no.~2, 455--482.

\bibitem{imgfact}
Ira~M. Gessel, \emph{A factorization for formal {L}aurent series and lattice
  path enumeration}, J. Combin. Theory Ser. A \textbf{28} (1980), no.~3,
  321--337.

\bibitem{gl}
Ira~M. Gessel and Ji~Li, \emph{Compositions and {F}ibonacci identities}, J.
  Integer Seq. \textbf{16} (2013), no.~4, Art. 13.4.5, 16 pp.

\bibitem{ly}
Jacques Labelle and Yeong~Nan Yeh, \emph{Generalized {D}yck paths}, Discrete
  Math. \textbf{82} (1990), no.~1, 1--6.

\bibitem{cristobal}
Cristobal Lemus-Vidales, \emph{Lattice {P}ath {E}numeration and
  {F}actorization}, Ph.D. Thesis, Brandeis University, 2017.

\bibitem{mrsv}
Donatella Merlini, D.~G. Rogers, Renzo Sprugnoli, and M.~Cecilia Verri,
  \emph{Underdiagonal lattice paths with unrestricted steps}, Discrete Appl.
  Math. \textbf{91} (1999), no.~1-3, 197--213.

\bibitem{rubine}
Dean Rubine, \emph{Hyper-{C}atalan and {G}eode recurrences and three
  conjectures of {W}ildberger}, \arXiv{2507.04552}{math.CO}.

\bibitem{schutzenberger}
Marcel~Paul Sch\"utzenberger, \emph{Une th\'eorie alg\'ebrique du codage}, C.
  R. Acad. Sci. Paris \textbf{242} (1956), 862--864.

\bibitem{stanley}
Richard~P. Stanley, \emph{Catalan {N}umbers}, Cambridge University Press, New
  York, 2015.

\bibitem{wr}
N.~J. Wildberger and Dean Rubine, \emph{A hyper-{C}atalan series solution to
  polynomial equations, and the {G}eode}, Amer. Math. Monthly \textbf{132}
  (2025), no.~5, 383--402.

\end{thebibliography}
\bibliographystyle{amsplain}

\end{document}